\documentclass{amsart}
\usepackage{amssymb,euscript,amsmath, mathrsfs}
\usepackage[dvips]{graphicx}
\usepackage[dvips]{color}

\newcounter{ENUM}

\input xy
\xyoption{all}
\CompileMatrices

\def\<{\langle}
\def\>{\rangle}
\def\0{{{\bf 0}}}
\def\CD{{\mathcal D}}
\def\OO{{\mathcal O}}

\def\CT{{\mathcal T}}

\def\CV{{\mathcal V}}

\def\AA{{\mathbb A}}
\def\CC{{\mathbb C}}

\def\FF{{\mathbb F}}
\def\GG{{\mathbb G}}

\def\QQ{{\mathbb Q}}
\def\RR{{\mathbb R}}

\newcommand{\Lie}{\operatorname{Lie}}

\newcommand{\Res}{\operatorname{Res}}

\def\ZZ{{\mathbb Z}}

\def\Hom{\operatorname{Hom}}

\def\End{\operatorname{End}}

\def\Spec{\operatorname{Spec}}

\def\Res{\operatorname{Res}}

\newcommand{\margh}[1]{}

\newtheorem{thm}{Theorem}[section]

\newtheorem{lemma}[thm]{Lemma}
\newtheorem{cor}[thm]{Corollary}

\theoremstyle{definition}

\newtheorem{rem}[thm]{Remark}

\numberwithin{equation}{section}

\begin{document}
\title[An ordinary abelian variety with an {\'e}tale self-isogeny]{An ordinary abelian variety with an {\'e}tale self-isogeny of $p$-power degree and no isotrivial factors}
\author{David Helm}

\maketitle

\begin{abstract}
We construct, for every prime $p$, a function field $K$ of characteristic $p$ and an ordinary abelian variety $A$ over $K$, with no isotrivial factors, that admits an {\'e}tale self-isogeny
$\phi: A \rightarrow A$ of $p$-power degree.  As a consequence, we deduce that there exist ordinary abelian varieties over function fields whose groups of points over the maximal purely
inseparable extension is not finitely generated, answering in the negative a question of Thomas Scanlon.
\end{abstract}

\section{Introduction}

Motivated by Hrushovski's celebrated work on the Mordell-Lang conjecture~\cite{Hr}, Thomas Scanlon posed the following question: given an abelian variety $A$ over a function field $K$
of characteristic $p$ with no isotrivial factors, let $K^{p^{-\infty}}$ denote the maximal purely inseparable extension of $K$.  Is the group $A(K^{p^{-\infty}})$ necessarily finitely generated?

Although this question is folklore and has not appeared explicitly in the literature, it and related questions have motivated considerable recent research.  For instance, the main results (Theorems 1.1 and 1.2)
of~\cite{Ro} establish conditions on $A$ under which $A(K^{p^{-\infty}})$ is provably finitely generated.  More recently, Ambrosi~\cite{Am} has given such criteria in terms of the action of $\End(A)$ on
the $p$-divisible group of $A$; in particular he shows that if $\End(A) \otimes \QQ_p$ is a division algebra then $A(K^{p^{-\infty}})$ is finitely generated.

This question is connected to the full Mordell-Lang conjecture in positive characteristic.  In particular if $A(K^{p^{-\infty}})$ is finitely generated, then the portion of the Mordell-Lang conjecture proved by
Hrushovski shows that for any closed subvariety $X$ of $A$ over $K$ that is not a translate of an abelian subvariety, we have that $X(K^{p^{-\infty}})$ is not dense in $X$.  The main result of
\cite{GhMo} then shows that the full Mordell-Lang conjecture holds for $A$.  (Note that the converse implication does not necessarily hold; in particular it is certainly possible for the full Mordell-Lang conjecture
to hold for $A$ even if $A(K^{p^{-\infty}})$ is not finitely generated!)

The goal of this paper is to answer Scanlon's question in the negative.  In fact we address a related question, namely that of constructing a function field $K$ of characteristic $p$,
and an ordinary abelian variety $A$ over $k$, with no isotrivial factors, such that $A$ admits a nontrivial {\'e}tale self-isogeny of $p$-power degree.  Indeed, we show:

\begin{thm} \label{thm:main}
For any prime $p$, there exists a global function field $K$ of characteristic $p$, an ordinary abelian variety $A$ over $K$, with no isotrivial factors, and a nontrivial
{\'e}tale self-isogeny $\phi: A \rightarrow A$ of $p$-power degree.
\end{thm}

It has been well-known for some time that this implies a negative answer to Scanlon's question.  In particular we have:
\begin{cor} \label{cor:infinite}
Let $A, K, \phi$ be as in Theorem~\ref{thm:main}.  Then for some finite extension $K'$ of $K$, the group $A^{\vee}((K')^{p^{-\infty}})$ is not finitely generated.
\end{cor}
\begin{proof}
Let $K'$ be a finite extension of $K$ such that every simple abelian subvariety of $A^{\vee}$ has a non-torsion point over $K'$.  The isogeny $\phi^{\vee}: A^{\vee} \rightarrow A^{\vee}$ is purely inseparable, 
so induces a surjection of $A^{\vee}((K')^{p^{-\infty}})$ onto itself.  Since this group contains a non-torsion point that is not fixed by $\phi^{\vee}$, 
and is infinitely $\phi^{\vee}$-divisible, it cannot be finitely generated.
\end{proof}

As we have remarked above, our result does {\em not} give a counterexample to the full positive characteristic Mordell-Lang conjecture.

In our construction, $A$ arises from a certain one-dimensional Shimura variety of Kottwitz-Harris-Taylor type.  More precisely, $A$ is obtained from the universal abelian scheme on this Shimura variety via
pullback to a generic point on the characteristic $p$ fiber.  The desired self-isogeny $\phi$ can then be constructed from the PEL structure on $A$.

{\em Acknowledgements:}
I am grateful to Felipe Voloch for first introducing me to this question, and to Damien R{\"o}ssler for reminding me of the question and explaining its history and background, as well as for his
enthusiasm and encouragement.  I would also like to thank H{\'e}l{\`e}ne Esnault, Marco D'Adezzio, and Emiliano Ambrosi for their helpful comments on this note.

\section{Unitary Shimura varieties} \label{sec:shimura}

The abelian variety $A$ of Theorem~\ref{thm:main} will be constructed from the universal abelian scheme on a certain Shimura variety arising from a unitary group.  We therefore briefly recall
the moduli definition of such Shimura varieties and their integral models.  We refer the reader to Kottwitz~\cite{Ko}, particularly sections 4 and 5, for further details.

Fix an imaginary quadratic field $E$ and a totally real field $F^+$, and let $F = EF^+$ be their compositum.  Also fix an embedding $\iota: E \rightarrow \CC$.  For each real embedding $\tau: F^+ \rightarrow \RR$, let
$p_{\tau}$ denote the embedding of $F$ in $\CC$ whose restriction to $F^+$ is $\tau$ and whose restriction to $E$ is $\iota$, and let $q_{\tau}$ denote its complex conjugate.

Let $n$ be a positive integer and let $\CV$ be an $n$-dimensional $F$-vector space, equipped with an alternating, non-degenerate $\QQ$-bilinear pairing:
$$\langle \cdot,\cdot \rangle: \CV \times \CV \rightarrow \QQ,$$
such that $\langle \alpha x, y \rangle = \langle x, \overline{\alpha} y \rangle$ for all $\alpha$ in $F$.  Let $G$ 
be the algebraic group over $\QQ$ whose $R$-points, for any $\QQ$-algebra $R$, are the $F \otimes_{\QQ} R$-linear automorphisms $g$ of $\CV \otimes_{\QQ} R$ 
that satisfy $\langle gx, gy \rangle = r \langle x,y \rangle$ for some $r \in R^{\times}$.)

The group $G$ is often called a ``unitary group over $\QQ$'', as (up to a similitude factor) $G(\RR)$ is isomorphic to a product of unitary groups.  More precisely,
for each $\tau$, the pairing $\langle \cdot,\cdot \rangle$ is the imaginary part of a unique Hermitian pairing on the $n$-dimensional $\CC$-vector space $\CV_{\tau}: = \CV \otimes_{F^+,\tau} \RR$.  Let the pair
$(r_{\tau},s_{\tau})$ denote the signature of this pairing.  Then we have an embedding:
$$G(\RR) \rightarrow \prod_{\tau} \operatorname{GU}(r_{\tau},s_{\tau}),$$
whose image is those elements of the product for which the ``similitude factor'' is independent of $\tau$.

The action of $E$ on $\CV$ gives rise to a map $h: \Res_{E/\QQ} \GG_m \rightarrow G,$ and if $h_{\RR}$ is the base change of $h$ to $\RR$, then the pair $(G,h_{\RR})$ is a Shimura datum on $G$.  In particular,
if $X$ denotes the $G(\RR)$-conjugacy class of $h_{\RR}$, then for any compact open subgroup $U$ of $G(\AA^{\infty}_{\QQ})$, the double quotient $G(\QQ) \backslash G(\AA^{\infty}_{\QQ}) \times X/U$ has the
natural structure of an algebraic variety over $\CC$, called the unitary Shimura variety attached to this Shimura datum.

Fix a prime $p$ that is unramified in $F$ and split in $E$, and an embedding of $\overline{\QQ}_p$ in $\CC$.  Then each embedding $p_{\tau}$ or $q_{\tau}$ of $F$ in $\CC$ gives rise to an embedding
of $F$ in $\overline{\QQ}_p$, and hence also a map $\OO_F \rightarrow \overline{\FF}_p$.  We somewhat abusively use $p_{\tau}$ and $q_{\tau}$ to denote these latter maps as well.
Once we have fixed this choice, the above Shimura variety admits (for suitable subgroups $U$)
a canonical model over $W(\overline{\FF}_p)$ that admits a natural description in terms of moduli of abelian varieties with certain PEL structure.

To describe this model, and its associated moduli problem, we first fix a ${\hat Z}$-lattice
${\hat \CT}$ in $\CV \otimes_{\QQ} {\mathbb A}_{\QQ}^{\infty}$, stable under the action of $\OO_F$, such that the cokernel of the induced map ${\hat \CT} \rightarrow \Hom({\hat \CT}, \hat \ZZ)$ induced by the pairing
$\langle \cdot,\cdot \rangle$ has order prime to $p$.  Let $U_p$ be the subgroup of $G(\QQ_p)$ preserving ${\hat \CT}_p$, and let $U$ be any compact open subgroup of $G(\AA^{\infty}_{\QQ})$ of the form
$U_p U^p$, where $U^p$ is a sufficiently small compact open in $G(\AA^{\infty,p}_{\QQ})$ that preserves ${\hat \CT}^{(p)}$.  
Then the Shimura variety $Y_U$ whose complex points are the double quotient $G(\QQ) \backslash G(\AA^{\infty}_{\QQ}) \times X /U$
admits an integral model over $W(\overline{\FF}_p)$, such that for any $W(\overline{\FF}_p)$-algebra $R$, the $R$-points of $Y_U$ are isomorphism classes of tuples $(A,\lambda,\rho)$, where:

\begin{enumerate}
\item $A$ is an abelian scheme over $R$ of dimension $nd$, where $d = [F^+:\QQ]$, with an action of $\OO_F$, such that for all $\alpha \in \OO_F$, the characteristic polynomial of $\alpha$, considered as an endomorphism
of $\Lie(A/R)$, is given by:
$$\prod_{\tau} (X - p_{\tau}(\alpha))^{r_{\tau}}(X - q_{\tau}(\alpha))^{s_{\tau}},$$
\item $\lambda$ is a polarization of $A$, of degree prime to $p$, such that the Rosati involution associated to $\lambda$ induces complex conjugation on $\OO_F \subseteq \End(A)$, and
\item $\rho$ is a ``$U$-level structure'' on $A$; that is, a $U$-orbit of isomorphisms (of {\'e}tale sheaves on $\Spec R$) from the constant sheaf ${\hat T}^{(p)}$ to the product
$$\prod_{\ell \neq p} T_{\ell} A$$
of the $\ell$-adic Tate-modules of $A$, compatible with the action of $\OO_F$, and the pairings on ${\hat T}^{(p)}$ and $T_{\ell} A$ for all $\ell$.
\end{enumerate}

The scheme $Y_U$ is then smooth over $W(\overline{\FF}_p)$ of relative dimension $\sum_{\tau} r_{\tau} s_{\tau}$.

Note that for any $W(\overline{\FF}_p)$-algebra $R$, and any $R \otimes_{\ZZ} \OO_F$-module $M$, we have a natural decomposition:
$$M \cong \bigoplus_{\tau} M_{p_{\tau}} \oplus M_{q_{\tau}},$$
where $M_{p_{\tau}}$ (resp. $M_{q_{\tau}}$) is the submodule of $M$ on which $\OO_F$ acts via the map $p_{\tau}: \OO_F \rightarrow W(\overline{\FF}_p)$ (resp. $q_{\tau}: \OO_F \rightarrow W(\overline{\FF}_p)$).
In particular, if $x$ is an $R$-point of $Y_U$, and $A_x$ is the corresponding moduli object, then $\Lie(A_x/R)$ admits such a decomposition.  The condition on the characterstic polynomial of the action of
$\OO_F$ on $\Lie(A_x/R)$ then shows that for each $\tau$, $\Lie(A_x/R)_{p_{\tau}}$ is a locally free $R$-module of rank $r_{\tau}$, and $\Lie(A_x/R)_{q_{\tau}}$ is likewise locally free over $R$ of rank $s_{\tau}$.

This has useful implications for the structure of the Dieudonn{\'e} modules of moduli objects for $Y_U$.
Let $k$ be an algebraically closed field of characteristic $p$, and $x$ a $k$-point of $Y_U$; then we denote by $\CD(A_x)$ the contravariant Dieudonn{\'e} module of the $p$-divisible group of $A_x$.  It is
a free $W(k)$-module of rank $2nd$, equipped with two natural endomorphisms $F$ and $V$, such that $FV= VF = p$, and
such that for any element $r$ of $W(k)$, and any $d \in \CD(A_x)$, we have $F(rd) = r^{\sigma} F(d)$ and $V(r^{\sigma}d) = rV(d)$, where $\sigma$ is the Frobenius automorphism of $W(k)$.
Moreover, we have a natural isomorphism $\CD(A_x)/p\CD(A_x) \cong H^1_{\mbox{DR}}(A_x/k)$, and under this isomorphism, the subspace $\Lie(A_x/k)^{\vee}$ of $H^1_{\mbox{DR}}(A_x/k)$ is
identified with the image of $V$ in $\CD(A_x)/p\CD(A_x)$.

The $\OO_F$-action on $A_x$ induces a corresponding action on $\CD(A_x)$, and hence we have summands $\CD(A_x)_{p_{\tau}}$ and $\CD(A_x)_{q_{\tau}}$ of $\CD(A_x)$ for all $\tau$.  The polarization $\lambda$
identifies $\CD(A_x)_{p_{\tau}}$ with the $W(k)$-dual of $\CD(A_x)_{q_{\tau}}$; this duality interchanges the endomorphisms $F$ and $V$.

Each summand $\CD(A_x)_{p_{\tau}}$ and $\CD(A_x)_{q_{\tau}}$ is a free $W(k)$-module of rank $n$.
The semilinearity properties of $F$ and $V$ imply that $F$ maps the summand $\CD(A_x)_{p_{\tau}}$ to $\CD(A_x)_{p_{\sigma \tau}}$, where $\sigma \tau$ is the composition of $\tau: \OO_{F^+} \rightarrow W(\overline{\FF_p})$
with the Frobenius automorphism of $W(\overline{\FF}_p)$.  Similarly $V$ maps $\CD(A_x)_{p_{\sigma\tau}}$ to $\CD(A_x)_{p_{\tau}}$.  The identification of the image of $V$ (modulo $p$) with the dual of $\Lie(A_x/k)$
then shows that the quotient $\CD(A_x)_{p_{\tau}}/V \CD(A_x)_{p_{\sigma\tau}}$ has dimension $n - r_{\tau} = s_{\tau}$.

\section{The $U(2)$ Shimura curve}
We now specialize to the particular case of interest to us.  Fix a prime $p$, and choose $E$ and $F^+$ so that $p$ splits in both, with $[F^+:\QQ] = 2$.  Let $\tau_0,\tau_1$ be the two real embeddings of
$F^+$; we will henceforth write $p_0,q_0,p_1,q_1$ for $p_{\tau_0},q_{\tau_0},$ and so forth.  Let $\pi_0$ and $\pi_1$ be the kernels of the maps $p_0,p_1: \OO_F \rightarrow \overline{\FF}_p$, and let $\pi'_0$ and
$\pi'_1$ denote their complex conjugates; the ideal $p$ of $\OO_F$ then factors as $\pi_0 \pi_1 \pi_0' \pi_1'$.  Finally, we choose $\CV$ so that $(r_0,s_0) = (2,0)$ and $(r_1,s_1) = (1,1)$.
(The existence of such a $\CV$ follows, for instance, from Lemma 1.1 in \cite{HT}, taking the division algebra $B$ to be a matrix algebra.)

In this case, for suitable $U$, the Shimura variety $Y_U$ is a Shimura variety of ``Kottwitz-Harris-Taylor type''.  Moreover, it is smooth of dimension $2\cdot 0 + 1\cdot1 = 1$; 
in particular its fiber over $\overline{\FF}_p$ is a smooth (not necessarily connected) curve.  Moreover, we have:

\begin{lemma} Under the above assumptions, $Y_U$ is proper over $W(\overline{\FF}_p)$.
\end{lemma}
\begin{proof}
Let $R$ be a discrete valuation ring over $W(\overline{\FF}_p)$, with field of fractions $K$, and let $(A,\lambda,\rho)$ be a $K$-point of $Y_U$.  It suffices to show that $A$ has potentially good reduction, i.e. that
there exists a finite extension $K'$ of $K$,
such that $A_{K'}$ has good reduction.  Indeed, if this holds then $A$ extends to an abelian scheme over the integral closure $R'$ of $R$ in $K'$, and the $\OO_F$-action, as well as the
structures $\lambda$ and $\rho$ likewise extend.  Thus $Y_U$ is proper by the valuative criterion.  

Suppose that $A$ does not have potentially good reduction.  Then $A$ has potentially multiplicative reduction over $K$, so there exists a finite extension $K'$ of $K$ such that $A$ extends to a semiabelian scheme over $R'$ whose special
fiber is the extension of an abelian variety by a torus.  Let $X$ be the character group of the toric part of this special fiber.  Then $\OO_F$ acts on $X$, so $X$ has rank at least four as a $\ZZ$-module.  Since
$A$ has dimension four, it follows that the special fiber is purely toric, and that $X$ is a locally free $\OO_F$-module of rank one.  But then it would follow immediately that the Lie algebra of this special
fiber is free of rank one over $\OO_F \otimes R/{\mathfrak m}_R$, and hence we would have $r_0 = s_0 = r_1 = s_1 = 1$, contrary to our assumptions.
\end{proof}

Now let $k$ be an algebraically closed field of characteristic $p$, and $x$ a $k$-point of $Y_U$.  We then have a decomposition:
$$\CD(A_x) \cong \CD(A_x)_{p_0} \oplus \CD(A_x)_{p_1} \oplus \CD(A_x)_{q_0} \oplus \CD(A_x)_{q_1},$$
with each summand a free $W(k)$-module of rank two.  Since $p$ splits completely in $F$ each summand is stable under $F$ and $V$, and in fact $\CD(A_x)_{p_i}$ is the (contravariant) Dieudonn{\'e} module
of the $p$-divisible group $A[\pi_i^{\infty}]$ for $i=0,1$.  Similarly $\CD(A_x)_{q_i}$ is the Diedudonn{\'e} module of $A_x[(\pi'_i)^{\infty}]$, which is Cartier dual to $A_x[\pi_i^{\infty}]$.  Since $(r_0,s_0) = (2,0)$,
we see that $F$ is an isomorphism on $\CD(A_x)_{p_0}$ and $V$ is an isomorphism on $\CD(A_x)_{q_0}$, so that $A_x[\pi_0^{\infty}]$ and $A_x[(\pi'_0)^{\infty}]$ are {\'e}tale and multiplicative, respectively.
By contrast, since $(r_1,s_1) = (1,1)$, both $A_x[\pi_1^{\infty}]$ and its Cartier dual $A_x[(\pi'_1)^{\infty}]$ are isomorphic to the $p$-divisible group of some elliptic curve.

There are thus only two possibilities for the Newton polygon of $A[p^{\infty}]$, depending on whether the $p$-divisible group $A[\pi_1^{\infty}]$ is that of an ordinary or supersingular elliptic curve.  In the
former case $A_x$ is an ordinary abelian variety; in the latter $A_x$ has $p$-rank two.  The ordinary locus in $(Y_U)_{\overline{\FF}_p}$ is open and dense, and the $p$-rank two locus is a proper closed subvariety, which
is necessarily zero-dimensional.  Moreover, it follows from general results on non-emptiness of Newton strata for Shimura varieties that the $p$-rank two locus is nonempty; in the PEL case under consideration this 
is Theorem 1.6 of Viehmann-Wedorn~\cite{VW13}; for abelian varieties of Hodge type this is a recent result of Kisin, Madapusi Pera, and Shin~\cite{KMPS}.  (It is likely that in the particular case under consideration 
one could prove this nonemptiness directly, and more simply, via Honda-Tate theory, but we do not attempt this.) 

\section{Proof of Theorem~\ref{thm:main}} \label{sec:main}
Fix a connected component $C$ of $(Y_U)_{\overline{\FF}_p}$ that contains at least one $\overline{\FF}_p$-point $x$ such that $A_x$ has $p$-rank two.  Let $\eta$ denote the generic point of $C$, let $K$ denote
the residue field of $\eta$, and let $A$ be the base change to $K$ of the universal abelian scheme on $Y_U$.  Since the ordinary locus is open and dense in $C$, $A$ is an ordinary abelian variety over $K$.

Moreover, it is not hard to construct an {\'e}tale self-isogeny of $A$ of $p$-power degree.  Indeed, for some $r$ the ideal $\pi_1^r$ is a principal ideal of $\OO_F$, generated by an element $\alpha$ of $\OO_F$.
Then the kernel of multiplication by $\alpha$ is a subgroup of the $p$-divisible group $A[\pi_1^{\infty}]$, and our Dieudonn{\'e} theory calculations show that
this $p$-divisible group is {\'e}tale.

We are thus reduced to showing:

\begin{lemma} The abelian variety $A/K$ has no isotrivial factors.
\end{lemma}
\begin{proof}
We will show that there is no pair $(B,\phi)$ with $B$ isotrivial over $K$ and $\phi: B \rightarrow A$ a morphism of abelian varieties with finite kernel.  Suppose that such a pair exists.  Since $A$ is ordinary,
$B$ must be ordinary as well.  Moreover, it is not hard to see that $B$ has good reduction at every point $y$ of $C$: this is certainly true for the image of $B$ in $A$, as $A$ has everywhere good
reduction; since $B$ is isogenous to its image in $A$ we deduce that $B$ has good reduction everywhere as well.  As $B$ is isotrivial and ordinary, the reduction $B_y$ must also be ordinary.

Consider the natural map 
$$\phi \otimes_{\ZZ} \OO_F: B \otimes_{\ZZ} \OO_F \rightarrow A.$$
As $B \otimes_{\ZZ} \OO_F$ has ordinary reduction everywhere, the image of this map is an $\OO_F$-stable subvariety $B'$ of $A$ with everywhere ordinary reduction.  In particular $B'_x$ is an ordinary, $\OO_F$-stable
abelian subvariety of $A_x$.  Since $A_x[\pi_1^{\infty}]$ is the $p$-divisible group of a supersingular elliptic curve, we have $B'_x[\pi_1^{\infty}] = 0$.  But then $\End_{\OO_F}(B'_x)$ is a 
torsion-free, finite rank $\OO_F$-module whose completion at $\pi_1$ vanishes.  We thus conclude that $\End_{\OO_F}(B'_x) = 0$, and hence that $B'_x = 0$.  Thus $B' = 0$ and therefore $B=0$ as well.
\end{proof}

\begin{rem} Since the Shimura variety $Y_U$ is proper, the abelian variety $A$ has everywhere good reduction.  Note that this is consistent with Theorem 1.2 (a) of~\cite{Ro}.
\end{rem}


\begin{thebibliography}{}
\bibitem[Am21]{Am}
E. Ambrosi, {\em Perfect points of abelian varieties}, preprint, arXiv:2103.16568.

\bibitem[GM06]{GhMo}
D. Ghioca and R. Moosa, {\em Division points on subvarieties of isotrivial semi-abelian varieties}, Int. Math. Res. Not. 2006, Article ID 65437, 1--23.

\bibitem[Hr96]{Hr}
E. Hrushovski, {\em The Mordell-Lang conjecture for function fields,} Jour. of the AMS 9 (1996), 667--690.

\bibitem[HT01]{HT}
M. Harris and R. Taylor, {\em The geometry and cohomology of some simple Shimura varieties,} Ann. of Math. Studies 151, Princeton University Press (2001).

\bibitem[KMPS20]{KMPS}
M. Kisin, K. Madapusi Pera, and S. W. Shin, {Honda-Tate theory for Shimura varieties,} preprint, 2020.

\bibitem[Ko92]{Ko}
R. Kottwitz, {\em Points on some Shimura varieties over finite fields,} Jour. of the AMS 5 (1992), 373--444.

\bibitem[R{\"o}17]{Ro}
D. R{\"o}ssler, {\em On the group of purely inseparable points of an abelian variety defined over a function field of positive characteristic II}, 
Algebra and Number Theory 14 (2020), no. 5, 1123--1173.

\bibitem[Sc05]{Sc}
T. Scanlon, {\em A positive characteristic Manin-Mumford theorem}, Comp. Math. 141 (2005), no. 6, 1351--1364.

\bibitem[VW13]{VW13}
E. Viehmann and T. Wedhorn, {\em Ekedahl-Oort and Newton strata for Shimura varieties of PEL type}, Math. Ann. 356 (2013), no. 4, 1493--1550.
\end{thebibliography}
\end{document}